\documentclass[11pt]{article}

\catcode`\@=11
\catcode`\@=12

\usepackage{amsmath}
\usepackage{amssymb}
\usepackage{amsthm}
\usepackage{pstricks}
\usepackage[utf8]{inputenc}\usepackage[T1]{fontenc}

\usepackage{xcolor}
\usepackage{amsmath,amsfonts,amssymb,mathtools,multicol}
\usepackage{mathrsfs,bbm}
\usepackage{hyperref}
\usepackage{enumitem}

\theoremstyle{plain}
\newtheorem{lemma}{Lemma}[section]
\newtheorem{theorem}[lemma]{Theorem}
\newtheorem{corollary}[lemma]{Corollary}

\theoremstyle{definition}
\newtheorem{definition}[lemma]{Definition}
\newtheorem{remark}[lemma]{Remark}

\numberwithin{equation}{section}

\newtheorem*{hyp*}{Hypothesis}
\newtheorem*{theorem*}{Theorem}
\newtheorem*{lemma*}{Lemma}
\newtheorem*{proposition*}{Proposition}
\newtheorem*{defn*}{Definition}
\newtheorem*{corollary*}{Corollary}
\newtheorem*{remark*}{Remark}
\newtheorem*{example*}{Example}
\newtheorem*{whythm*}{Whyburn's Theorem}

\newcommand{\diam}{\text{diam}\,}
\newcommand{\dist}{\text{{\rm dist}}\,}

\newcommand{\ol}{\overline}
\newcommand{\bs}{\boldsymbol}
\renewcommand{\forall}{\text{ for all }}
\def\NN{\mathbb N}

\newcommand{\x}{\times}

\renewcommand{\leq}{\leqslant}
\renewcommand{\geq}{\geqslant}

\newcommand{\e}{\epsilon}

\newcommand{\sF}{\mathscr F}

\renewcommand{\a}{\alpha}
\newcommand{\mM}{\mathcal{M}}
\renewcommand{\d}{\delta}

\newcommand{\sN}{\mathscr N}

\newcommand{\sR}{\mathscr{R}}

\newcommand{\sP}{\mathscr{P}}

\newcommand{\sQ}{\mathscr{Q}}

\newcommand{\sO}{\mathscr{O}}

\renewcommand{\b}{\beta}
\newcommand{\g}{\gamma}

\newcommand{\mO}{\mathcal{O}}
\newcommand{\mF}{\mathcal{F}}

\renewcommand{\iff}{\text{ if and only if }}

\newcommand{\ed}{\end{document}}

\title{Peano  Quotients of Metric Continua}
\author{J. F. Toland\footnote{\noindent Department of Mathematical Sciences, University of Bath, Claverton Down, BA2 7AY, U.K. masjft@bath.ac.uk}}
\date{}

\begin{document}
\maketitle

\begin{abstract} \noindent
For any compact, connected metric space $(M,d)$   the set of points where $M$ is not weakly locally connected is shown to define  a partition $\sP$ of $M$ for which  the corresponding quotient metric space  $(\sQ, \nabla_\sQ)$  is a Peano continuum with $\sQ = \sP$.
\end{abstract}


\section{Introduction}

A metric space $(M,d)$ which is not empty, compact and connected is called a continuum and a continuum is a  Peano continuum if it is  locally connected. Hence, since a metric space  is locally connected if and only if it is weakly locally connected (Definition \ref{loccon}, see \cite[Thm.\,3\,-11]{hy} or  \cite[Lem.\,4.23]{bt})  a continuum  is a Peano continuum  if and only if the set  $\sN(M)$  of points where $M$ is not weakly locally connected is empty. 

\begin{subequations}\label{F&O} More precisely, for   an arbitrary  continuum $M$ let 

\begin{equation}\mF = \ol{\sN(M)}, ~~~\mO = M \setminus \mF,\end{equation}
and, with $\sF$ denoting the set of components of $\mF$ and  $\sO$  the set of singletons $\big\{\{x\}: x \in \mO\big\}$, let 
 \begin{equation}\text{$\sP = \sF \cup \sO$, which is a partition of $M$ and $M= \mF \cup \mO$}.\end{equation} 
Then if $(\sQ, \nabla_\sQ)$ denotes the quotient metric space \cite[Defn.\,3.1.12]{bby}  of $(M,d)$  generated by the partition  $\sP$ it is shown that
\end{subequations}
\begin{itemize}
\item[(i)]  $\sQ = \sP$ and $(\sQ, \nabla_\sQ)$  is a Peano continuum, Theorem \ref{P=Q} and Theorem\ref{Peanoresult};
\item[(ii)] $\sF$ is totally disconnected in $(\sQ, \nabla_\sQ)$, Theorem \ref{Fdiscon}; 
\item[(iii)] the quotient metric $\nabla_\sQ$ restricted to $\sO$ coincides locally with the original metric $d$ restricted to $\mO$ in the following sense: for every $x \in \mO$ there is a neighbourhood $U_x \subset \mO$ such that, Theorem \ref{L0w},
   $$ d(x_1,x_2) = \nabla_\sQ  \big(\{x_1\},\{x_2\}\big) \forall x_1,\,x_2 \in U_x.\qquad\qquad \Box$$\end{itemize}
If $\mF = \emptyset$, $\big(\sQ, \nabla_\sQ\big)$ is isometric to $(M,d)$, and if  $\mF= M$, as it is when $M$ is indecomposable, see \cite[Thms. 6.5 and 6.23]{bt}, $\sP =  \sQ =\{M\}$, a singleton. 

  Recall that by  the Hahn–Mazurkiewicz Theorem \cite{hahn,maz} a metric space $(\mM, \rho)$ is the continuous image of a closed interval if and only if it is a Peano continuum. Moreover, a continuum $(\mM,\rho)$ is Peano if and only if it admits an equivalent metric $\varrho$ such that, for all $x \in M$ and $\d>0$,
  $$ 
  \ol{\{y \in \mM: \varrho (x,y)< \d\}} = \{y \in \mM: \varrho (x,y)\leq \d\}, ~ \text{\rm see \cite{fraser, nadler}.}
  $$

Section \ref{closing} considers a  generalised version and the limitations of these results.

\textbf{Acknowledgement.}  When I asked  Logan Hoehn (Nipissing University, Canada) whether these results are already known, as seemed likely, he pointed to  \cite[Thm.\,3, p.\,247]{kuratowski} in which  Kuratowski cited a  somewhat neglected paper of R. L. Moore \cite{moore}\footnote{According to MathSciNet \cite{moore} is cited in only two papers, both involving the same author.} who was motivated by the same issues.  In \cite{moore}, a hundred years ago,    Moore developed his theory from first principles but  presented his results in  terminology that might today appear archaic. 
\qed    
      
\section{Terminology and Notation}\label{TN}
 There follows a summary of notation and some relevant  metric-space theory.
 
  A set  is a singleton if it has only one point and   non-degenerate if it has more than one point. 
In a metric space a set $A$ is totally disconnected if all its components are singletons.
\begin{remark}\label{discon} If for all   $a,b \in  A \subset M$ with $a \neq b$, 
there are closed sets $V$ and $W$ in $A$ with
  $
  \text{  $V \cup W =  A$,~~ $V \cap W = \emptyset$,~~ $b \in W$ and   $a \in V$}$, then $A$ is totally disconnected.\qed
\end{remark}

\begin{definition} \label{loccon}  A metric space $(M,d)$ is weakly locally connected \cite[\S4.5]{bt} at $x$   if for all $\e>0$ there is a connected neighbourhood $U$  of $x$ with $\diam (U) < \e$, and locally connected  \cite[\S 4.2]{bt}  at $x$ if there is an open connected neighbourhood $U$  of $x$ with $\diam (U) < \e$.   
 Then $M$ is said to be locally (or weakly locally) connected if it is locally (or weakly locally) connected at all $x \in M$.\qed  \end{definition}

     Obviously $M$ is weakly locally connected at  $x$ if it is locally connected at $x$, but if $M$ is not locally connected at $x$ it can still be weakly locally connected at $x$, for an example see \cite[Rmk. 4.22]{bt}.  
   However, a metric space $M$ is  locally connected if and only if $M$ is weakly locally connected \cite[Lem. 4.23]{bt}, and such a metric space is said to be a Peano metric space.

  If, for convenience as  in  \cite[Ch. 5]{bt}, a point  where $M$ is not weakly locally connected is referred to as  a congestion point\footnote{Moore \cite[p. 307]{moore} referred to points where $M$ is weak locally connected as points where $M$ is ``connected im kleinen'', and he referred to our  congestion points as ``irregular points''.},
  $\sN(M)$  in \eqref{F&O} denotes the set of  congestion points of $(M,d)$  and a continuum is a Peano continuum if and only if  $\sN(M) = \emptyset$. The following classical results are important. 
\begin{theorem}\label{whyburn}    If $M$ is a continuum, no component of $\sN(M)$ is a singleton. 
\\{ Proof. \rm See \cite[p. 18]{whyburn}, \cite[p. 102]{wilder} and \cite[Thm. 5.8]{bt}, and references therein. \qed}\end{theorem}

 \begin{theorem} \label{whythm} Suppose  $R$ and $S$ are non-empty, disjoint, closed subsets of a compact set  $H$  in a metric space $M$ and no component of $H$ intersects both $R$ and $S$. Then  there exist closed sets $H_r,H_s$  with
\begin{equation*}  H = H_r \cup H_s,\quad
H_r\cap H_s = \emptyset,\quad R \subset H_r \text{ and }S \subset H_s. 
\end{equation*}
{Proof. \rm See \cite[(9.3), p. 12]{whyburnsthm} and \cite[Thm. 3.53]{bt}. \qed}
\end{theorem}
The rest of Section \ref{TN}  summarises material on quotient spaces from \cite[Ch.\,3]{bby}.

\subsection{Partitions and  Equivalent Relations}

A partition $\sP$ of a set  $M$ is a family  of mutually disjoint,  non-empty sets $P$  with $\cup_{P\in \sP} P = M$, and for a partition  $\sP$   an equivalence relation $\sim$  on  $M$ is defined  by
\begin{equation}\label{sim}\text{$x\sim y \iff x$  and $y$ belong to the same $P \in \sP$},\end{equation}
and  $\sP$ is the set of  equivalent classes of $\sim$.
\begin{definition}\label{refinement}
A partition $\sP$ of $M$ is a refinement of  a partition $\sR$ of $M$  if for all $P \in \sP$ there is $R \in \sR$ with $P \subset R$.  \qed
\end{definition}

When $M\neq \emptyset$ and $d:M \x M\to [0,\infty)$ satisfies, for all $x,y,z \in M$,
$$(\a) ~ d(x,y)= d(y,x);~~
(\b) ~ d(x,z) \leq d(x,y) + d(y,z);
~~(\g) ~ d(x,y) =0 \Leftrightarrow x=y,$$
$(M,d)$  is a metric space and  for  $X, Y \subset M$,
\begin{equation} \dist (X,Y): = \inf \{d(x,y): x \in X,\, y \in Y\}. \label{dist}\end{equation}

If  $\Delta: M\x M\to [0,\infty)$ satisfies ($\a$), ($\b$) and $\Delta(x,x)=0$ for all $x \in M$, but $\Delta(x,y)= 0$ does not  imply  $x = y$,   then $(M,\Delta)$ is called a pseudo-metric space.

Then for a  pseudo-metric space $(M,\Delta)$  an  equivalence relation $\approx$ is defined by
 \begin{equation}\label{erD}
 x \approx y \iff \Delta(x,y)=0, \quad x,y \in M,
 \end{equation}
 and  a metric $\nabla$ is defined on  the set  $\sR$ of equivalence classes  of $\approx$ by
 \begin{equation}\label{Del<}
  \nabla(R_1,R_2): = \Delta(r_1,r_2) \text { for arbitrary }r_1 \in R_1,~ r_2 \in R_2, ~~ R_1,~R_2 \in \mathscr R.
  \end{equation}

Since $\sR$ is a partition of $M$, a function  $\pi:M \to \sR$ is   defined by the formula
\begin{gather}\label{pi} x \in \pi(x) \in \sR, \quad x \in M,\intertext{and  by \eqref{Del<}}\label{picont}
\nabla (\pi(x),\pi(y)) = \Delta(x,y) \forall x,\,y \in M.
\end{gather}

\subsection{Partitioned Metric Spaces and Pseudo-metrics}

Now suppose  $(M,d)$ is a metric space, $\sP$ is a partition of  $M$ and  $\sim$ is  defined by \eqref{sim}.
\begin{definition}\label{pstring}  For   $x, y \in M$, a $\sP$-string  joining $x$ to $y$ is  a finite    family of pairs  $\{(x_i,y_i)\in M \x M:1\leq i \leq k\}$  with
\begin{gather}\label{Dxy1}  x_1= x,~y_k = y,\text{ and if } k\geq 2,~ y_{i-1}\sim x_{i}, i\in \{2, \cdots, k\}.
 \intertext{Then following \cite[p.\,62,\,Defn. 3.1.12]{bby} let  $\Delta_\sP: M \x M \to [0,\infty)$ be defined  by} \Delta_\sP(x,y) = \inf_{\mathfrak P(x,y)}\Big\{\sum_{i=1}^k d(x_i,y_i)\Big\},\quad x,\,y \in M,
\label{Dxy2}
\end{gather}
 where $\mathfrak P(x,y)$ is the set of all $\sP$-strings joining $x$ to $y$. \qed
\end{definition}

\begin{theorem}
$(M,\Delta_\sP)$ is a pseudo-metric space with, for $x,\,y$ and $z \in M$,
\begin{align}\label{7}\Delta_\sP (x,y) &\leq d(x,y),\\
\label{8}\Delta_\sP(x,y) &= 0 \text{  if $x \sim y$}, \\
\label{9}\Delta_\sP(x,z)&=\Delta_\sP(y,z)\text{ if $x \sim y$ and $z \in M$}.
\end{align}
\end{theorem}

\begin{proof} This is Exercise 3.1.13 in \cite{bby}.
By  definition
$\Delta_\sP(x,y) \geq 0,~  x,y \in M$, and  the $\sP$-string $\{(x_1,y_1)\} = \{(x,y)\}$ implies
\begin{equation}\label{less}
0\leq \Delta_\sP(x,y) \leq d(x,y) \forall x,y \in M,
\end{equation}
and hence
\begin{equation} \label{zero}
\Delta_\sP (x,x) = 0 \forall x \in M.
\end{equation}
Also since $d$ is a metric on $M$,   for any $\sP$-string joining  $x$ to $y$ in $M$  with $k \geq 2$,
\begin{align*}
\sum_{i=1}^k d(x_i,y_i) &= \sum_{i=1}^k d(y_i,x_i),~~  x_1= x,~y_k = y, ~~   y_{i-1}\sim x_{i},~  i\in \{2, \cdots, k\}
\\&=  \sum_{j=1}^k d(y'_j,x'_j),~~  y'_1= y,~x'_k = x, ~~   x'_{j-1}\sim y'_{j},~  j\in \{2, \cdots, k\},
\end{align*}
where    $(y'_j,x'_j) = (y_i,x_i)$,   $j = k+1-i$, $ 1\leq j\leq k$. Since, when $k=1$, $d(x_1,y_1) = d(y_1,x_1)$, this shows that
\begin{equation}\label{sym}  \Delta_\sP(x,y) = \Delta_\sP(y,x) \forall x,y \in M.\end{equation}
Now, for $x,y \in M$ and any $\e>0$, there exists a $\sP$-string joining $x$ to $y$ with
\begin{align*}
 \sum_{i=1}^k d(x_i,y_i)&\leq \Delta_\sP(x,y) +\e,~~ x_1 = x,~~ y_k = y.
\intertext{Also for  the same $\epsilon$ and $y$, and any $z \in M$,  there   is a $\sP$-string   $\{(\hat y_i,\hat z_i)\}$  with}
  \sum_{i=1}^\ell d(\hat y_i,\hat z_i)&\leq \Delta_\sP(y,z)+\e, ~~   \hat y_1= y,~~ \hat z_\ell = z.
\end{align*}

Since $y_k = y = \hat y_1$,  concatenating  these $\sP$-strings  yields a $\sP$-string joining $x$ to $z$ and so
\begin{align*}
\Delta_\sP(x,z) &\leq  \sum_{i=1}^k d(x_i,y_i) + \sum_{i=k+1}^{k+\ell} d(\hat  y_{i-k},\hat z_{i-k}) \leq \Delta_\sP(x,y) + \Delta_\sP (y,z)+2\e \end{align*}
 for all $\e>0$. Hence
\begin{equation}\label{triangle}
\Delta_\sP (x,z) \leq \Delta_\sP(x,y) + \Delta_\sP (y,z), \forall x,y,z \in M.
\end{equation}
Since, by \eqref{zero}, \eqref{sym} and \eqref{triangle}, $\Delta_\sP$ satisfies axioms ($\alpha$), ($\beta$)    and the `if' part of ($\gamma$) for a metric space, $(M,\Delta_\sP)$   is a pseudo-metric space satisfying \eqref{less}.
Now if   $x \sim y$ in \eqref{Dxy2} 
\begin{equation}\label{seven}\Delta_\sP(x,y) \leq d(x,x) + d(y,y) = 0,\end{equation}
which proves \eqref{8}, and \eqref{9} follow from \eqref{triangle}.
This completes the proof.\end{proof}

\begin{remark}\label{R5}  For the $\sP$-string in \eqref{Dxy1} which joins $x$ to $y$, let $P_i \in \sP$ be such that
$$
x=x_1\in P_1,~~ y=y_k \in P_{k+1} \text{ and } ~~\{y_i, x_{i+1}\} \subset P_{i+1},\quad 1 \leq i \leq k-1.
$$
Then since by \eqref{dist}  $ \dist(P_{i}, P_{i+1}) \leq d(x_i,y_i)$ when $x_i \in P_i$ and $y_i \in  P_{i+1}$, $1\leq i\leq k$,
it follows from \eqref{Dxy2} that
\begin{gather}\label{Dxy3} \Delta_\sP(x,y) = \inf_{\widetilde{\mathfrak P}(x,y)}\Big\{\sum_{j=1}^{k} \dist(P_j, P_{j+1})\Big\},
\\
 \text{where }\widetilde{\mathfrak P}(x,y)= \big\{\{P_1,\cdots, P_{k+1}\}\subset \sP\colon~ k \in \NN,~ x \in P_1, ~y \in  P_{k+1}\big\}.\quad \Box \notag\end{gather}
\end{remark}

\subsection{Quotient Spaces} 

\begin{definition}\label{QS} For a partition $\sP$  of a metric space $M$ and the pseudo metric $\Delta_\sP$ defined in \eqref{Dxy2}, let $\approx$ be the equivalent relation  on $M$ defined by  \eqref{erD}, namely 
$ x\approx y \iff \Delta_\sP (x,y) = 0,$
and let  $ \sQ$ denote the set of  equivalent classes of $\approx$. 
Then, as in \eqref{Del<},   for $ Q_1,\,Q_2 \in \sQ$ 
$$
\nabla_\sQ(Q_1,Q_2) = \Delta_{\sP}(q_1,q_2) \text{ defines a metric   }\nabla_\sQ \text{ on  }\sQ,
$$
and if,  as in \eqref{pi},   $ x \in \pi_\sQ(x) \in \sQ,~ x \in M$,
$$
\text{ by \eqref{picont} }\nabla_\sQ (\pi_\sQ(x),\pi_\sQ(y)) = \Delta_\sP(x,y)\leq d(x,y) \forall x,\,y \in M.
$$
Then by  \cite[p.\,62,\,Defn,\,3.1.12]{bby}, $({\sQ}, \nabla_\sQ)$ is the quotient metric space   defined by the partition $\sP$ of the metric space $(M,d)$. Note that even in very simple examples  ${\sQ}$ and  $\mathscr P$ may not coincide.  \qed\end{definition}

\begin{lemma}
(a)  Every $Q \in \sQ$ is closed in $M$.\\
(b)  $\sP$ is a refinement (Definition \ref{refinement}) of $\sQ$ with
\begin{equation}\label{nab2}
\nabla_\sQ\big(\pi_\sQ(x_1), \pi_\sQ(x_2)\big) = \Delta_{\sP}(x_1,x_2) \leq d(x_1,x_2),~ x_1,\, x_2 \in M.\end{equation}
(c) If $P_1,\,P_2\in\sP$ and $\dist (P_1,P_2)=0$ there exist $Q \in {\sQ}$ with $P_1\cup P_2 \subset Q$. \\
(d) If $Q \in {\sQ}$ and $P \cap Q = \emptyset$, $P \in \sP$, then $\ol P \cap Q = \emptyset$.
\end{lemma}
\begin{proof} (a)  Since $Q = \{y \in M: \Delta_\sP(x,y) = 0\}$ when $x \in Q \in \sQ$,
and since $y \mapsto \Delta_\sP(x,y)$, $y \in M$, is  continuous by \eqref{7} and \eqref{triangle},  $Q$ is closed in $M$.

(b) By  \eqref{8}  $\sP$ is a refinement of $\sQ$  and \eqref{nab2} follows from  \eqref{picont} and \eqref{7}. 

(c) By  hypothesis there exists $x_k \in P_1$ and $y_k \in P_2$ such that $ d(x_k,y_k) \to 0$ as $k \to \infty$. Then for $x \in P_1$ and $y \in P_2$, by \eqref{8}, \eqref{triangle} and \eqref{seven}
\begin{align*} \Delta_\sP (x,y) &\leq \Delta_\sP (x,x_k) + \Delta_\sP (x_k,y_k) + \Delta_\sP(y_k,y)\\& =  \Delta_\sP (x_k,y_k) \leq d(x_k,y_k) \to 0.
\end{align*}
Hence $x \approx y$ and so $P_1\cup P_2\subset Q$ for some $Q \in {\sQ}$. 

(d) If $P \cap Q = \emptyset$ where $P \in \sP$ and $Q \in {\sQ}$, it follows from (b) that $P \subset  Q'\neq Q$ for some $Q' \in {\sQ}$ with $Q \cap Q' = \emptyset$. Then $\ol P \subset  Q'$ since $Q' $ is closed,   and hence $\ol P \cap Q = \emptyset$. This completes the proof.
\end{proof}

\section{Special Partitions of $M$}\label{SPM}

This section investigates the properties of $\Delta_\sP$ when $(M,d)$ is a metric space, not necessarily compact,  $\mF$ is a closed subset of $M$, $\mO = M \setminus \mF$,  $\sO$  is the set of singletons $\big\{\{x\}: x \in \mO\big\}$, $\sF$ is the set of components of $\mF$,  and 
\begin{equation}\label{FO}\sP = \sF \cup \sO \text{ which is a partition of $M$ since $M= \mF \cup \mO$},\end{equation}  
and $(\sQ, \nabla_\sQ)$ denotes the quotient metric space  of $(M,d)$  generated by $\sP$.

 \begin{lemma}\label{L1} When  $x,\,y\in M$,
\begin{equation}\Delta_\sP (x,y)
\geq \min\Big\{\max\{\dist (x,\mathcal F),\,\dist (y,\mathcal F)\},\, d(x,y) \Big\}.\label{L1b}
\end{equation}
 \end{lemma}
 \begin{proof} First note that since $d$ and $\Delta_\sP$ are symmetric,  \eqref{L1b} will follow if 
 \begin{equation}
 \label{L1o} \Delta_\sP (x,y) \geq \min\big\{\dist (x,\mathcal F), d(x,y)\big\} \forall x,\,y \in M. \end{equation}

  Now if $x \in \mF$ or if $x=y$,  \eqref{L1o} holds  because $\Delta_\sP(x,y)  \geq 0$ and the right side is zero.  So for some $x,y \in M$ suppose that  $x \notin \mF$, $x \neq y$  and  \eqref{L1o} is false. Then
 $$0\leq \Delta_\sP (x,y) <  \eta:= \min\{\dist (x,\mathcal F), d(x,y)\} > 0,$$
and by Definition \ref{pstring}  there is a $\sP$-string joining $x$ to $y$  with  
\begin{equation}\label{yes}\sum_{i=1}^k d(x_i,y_i) < \eta.\end{equation}
In particular $d(x,y_1)<\eta \leq d(x,\mathcal F)$ and it follows that $y_1 \in \mathcal O$,
 and hence that $x_2 = y_1$  since $y_1 \sim x_2$  and $\{y_1\} \in \sO$. Then it follows from \eqref{yes} and the triangle inequality that
$$
d(x,y_2)  +  \sum_{i=3}^k d(x_i,y_i)\leq  d(x,y_1) +d(y_1,y_2) +  \sum_{i=3}^k d(x_i,y_i) < \eta,
$$
and, as previously it follows that  $y_2 \in \mathcal O$,  and hence  $x_3 = y_2$. Again   by  \eqref{yes},
\begin{align*}
d(x,y_3)  +  \sum_{i=4}^k d(x_i,y_i)&\leq  d(x,y_1) +d(y_1,y_2) +d(y_2,y_3)+  \sum_{i=4}^k d(x_i,y_i) < \eta.
\end{align*}
Repeating this argument finitely often yields the contradiction that $d(x,y) =d(x,y_k)<\eta <d(x,y)$.  This proves \eqref{L1o}, and \eqref{L1b} follows.
\end{proof}
\begin{corollary}\label{Cxy} When $x \in \mO$ let $\d_x = \dist(x,\mF)/3>0$. Then 
$$d(x_1,x_2) = \Delta_\sP(x_1,x_2) \forall x_1,x_2 \in B_{\d_x}(x)= \{y:d(x,y) < \d_x\}.$$  
\end{corollary}
\begin{proof} The result follows from    \eqref{7} and \eqref{L1b} since  when $x_1,x_2 \in  B_{\d_x}(x)$,   $\dist(x_i,\mF) \geq 2\d_x > d(x_1,x_2)$, $i = 1,2$. 
\end{proof}
 \begin{theorem} \label{Atlast}
 When $\sP$ is defined by \eqref{FO} and $x,y \in \mF$,
 \begin{equation}\label{Dxy4} 
 \Delta_\sP(x,y) = \inf_{\mathfrak F(x,y)}\Big\{\sum_{j=1}^{k} \dist(F_j, F_{j+1})\Big\},
 \end{equation}
 where $\mathfrak F(x,y)$ is the  family of finite sets $\{F_1,\cdots, F_{k+1}\}\subset \sF$ with $x \in F_1$ and $y \in  F_{k+1}$.
\end{theorem}
\begin{proof} 
If $x \in \hat F$, $y \in \check F$ where $\hat F, \check F \in \sF$, by  Remark \ref{R5}
\begin{equation*} \Delta_\sP(x,y) = \inf_{\widetilde{\mathfrak P}(x,y)}\Big\{\sum_{j=1}^{k} \dist(P_j, P_{j+1})\Big\},
 \end{equation*}
  where $\widetilde{\mathfrak P}(x,y)$ is the family of finite sets $\{P_1,\cdots, P_{k+1}\}\subset \sP$ with $x \in P_1=\hat F\in \sF$, $y \in  P_{k+1}= \check F\in \sF$. 
  Now suppose that  
  \begin{equation}\label{Dxy3}
  \inf_{\widetilde{\mathfrak P}(x,y)}\Big\{\sum_{j=1}^{k} \dist(P_j, P_{j+1})\Big\}<\eta,
  \end{equation}  
   and that for some   $P_{j'} \notin \sF$, $ 2\leq j'\leq k$. Then  there exist  $n, m \in \{ 1,\cdots, k+1\}$ with $j' \in \{n+1, \cdots, m-1\}$, such that $P_n=  F_n\in \sF$, $P_m  = F_m\in \sF$, and
   $P_j = \{y_j\} \in \sO$ for all $j \in \{n+1, \cdots, m-1\}$. 
   Therefore, for any $\d>0$ there exist $y_n \in F_n$ and $y_m \in F_m$ such that 
   \begin{multline*}\sum_{j=n}^{m-1} \dist(P_j, P_{j+1}) = \dist(F_n, y_{n+1})+\Big\{\sum_{j=n+1}^{m-2} d(y_{j}, y_{j+1})\Big\} + \dist(y_{m-1}, F_m)\\  \geq \sum_{j=n}^{m-1} d(y_{j}, y_{j+1})-\d/k \geq d(y_n,y_m)-\d/k \geq\dist(F_n,F_m)-\d/k.
  \end{multline*}
Hence 
    $\sum_{j=n}^{m} \dist(P_j, P_{j+1})$ in \eqref{Dxy3} can be replaced by $\dist(F_n, F_m)$ while increasing the value of the sum by at most $\d/k$.
          Since  $x \in F_1 \in \sF$ and $y \in F_{k+1} \in  \sF$, the preceding argument can be repeated at most $k$ times until only elements of $\sF$ remain, whence \eqref{Dxy3} implies
          $$ 
\Delta_\sP(x,y)     \leq      \inf_{\mathfrak F(x,y)}\Big\{\sum_{j=1}^{k} \dist(F_j, F_{j+1})\Big\}<\eta+\d
          $$ 
          and  \eqref{Dxy4} follows since $\d>0$ may be arbitrarily small. 
 \end{proof}

 \begin{corollary}\label{tildehat} When $\sP$ is defined by \eqref{FO} and $\mF$ is compact, $\Delta_\sP(x,y) > 0$ for  $x \in \hat F$ and $y \in \check F$ if $\hat F \neq \check F$ and $\hat F,\,\check F \in \sF$.
\end{corollary} 
\begin{proof} The   components $\hat F$ and $\check F$  of the compact set $\mF \subset M$ are closed and no component of $\mF$ intersects both $\hat F$ and $\check F$. Therefore, by Theorem \ref{whythm} there exist two closed sets $\hat H$ and $\check H$ with 
$$
\mF = \hat H \cup \check H,~ \hat H \cap \check H = \emptyset,~ \hat F \subset \hat H,~\check F \subset \check H,$$
 and  $ \dist(\hat H,\check H) >0$  by compactness.
Now suppose $\{F_1, \cdots, F_{k+1}\}$ is any finite family of sets in $\sF$ with $F_1 = \hat F$ and $F_{k+1} = \check F$. Then since all these sets   are  connected,  each  is a subset of  either $\hat H$ or $\check H$ and hence there exists $j \in \{1, \cdots, k\}$ such that $F_j \in \hat H$ and $F_{j+1} \in \check H$, and so, for all such families,
    $$\sum_{i=1}^{k} \dist (F_i,F_{i+1})  \geq  \dist (F_j,F_{j+1})\geq \dist(\hat H,\check H) >0.$$ Therefore, by Theorem \ref{Atlast},
     $\Delta_\sP(x,y) \geq \dist(\hat H,\check H) >0$ as required.
  \end{proof}
  \begin{theorem}\label{P=Q}
When $\sP$ is defined by \eqref{FO}, $\mF$ is compact and $(\sQ,\nabla_\sQ)$  denote the corresponding quotient metric space  (Definition \ref{QS}),  $\sP = \sQ$.
 \end{theorem}  
 \begin{proof}
 Since the aim is to show that $\Delta_\sP(x,y) = 0$ implies $\{x,y\} \subset P$ for some $P \in \sP$  note, from    Lemma \ref{L1}, that $\Delta_\sP(x,y) > 0$ if $x \neq y$ and  $\{x,y\} \not\subset \mF$ . Moreover, if $\{x,y\} \subset \mF$ with $x \in \hat F$ and $y \in \check F$, where $\hat F\neq\check F$ are elements of $\sF$, it follows  from Corollary \ref{tildehat} that $\Delta_\sP(x,y)>0$. Therefore,  $\Delta_\sP(x,y) =0$ and $x \neq y$ implies that both $x$ and $y$ are in the same $F \in \sF$, and the result follows since $\sF\cup \sO= \sP$.
 \end{proof}

\section{Peano Quotients of Continua}\label{PQMST}

This section deals with the particular case of Section \ref{SPM}  when $(M,d)$ is a continuum, the partition $\sP$ in \eqref{FO} is defined by \eqref{F&O} and  $(\sQ,\nabla_\sQ)$ denotes the corresponding quotient metric space.

\begin{theorem} \label{comcon}The metric space $\bs ({\sQ}, \nabla_\sQ)$ is a continuum.
\end{theorem}
\begin{proof} 
To show compactness of  $\bs ({\sQ}, \nabla_\sQ)$ it suffices to show that every sequence 
$\{Q_k\}\subset {\sQ}$ has a convergent subsequence in $({\sQ}, \nabla_\sQ)$. So suppose $q_k \in Q_k$ and, since $(M,d)$ is compact, without  loss of generality suppose  that $d(q_k,q) \to 0$ where $q \in Q \in {\sQ}$ because $\sQ$ is a partition of $M$.  It then follows from Definition \ref{QS} that $({\sQ}, \nabla_\sQ)$ is compact since
$$
\nabla_\sQ(Q_k, Q) = \Delta_\sP(q_k,q) \leq d(q_k,q) \to 0.
$$

Since by \eqref{nab2}
$$
\nabla_\sQ (\pi_\sQ(x),\pi_\sQ(y)) = \Delta_\sP(x,y)\leq d(x,y) \forall x,\,y \in M,
$$
$\pi_\sQ:M \to \sQ$ is a continuous surjection and so $({\sQ}, \nabla_\sQ)$ is connected because $(M,d)$ is connected. This completes the proof
\end{proof}

  \begin{theorem}\label{Fdiscon}
The metric space $(\sF,\nabla_\sQ)$   is compact and totally disconnected.
\end{theorem}
\begin{proof} Since $\sF = \pi_\sQ(\mF)$ where $\pi_\sQ: (M,d) \to (\sQ, \nabla_\sQ)$ is continuous and $\mF$ is compact, it follows that $\sF$ is compact in $(\sQ, \nabla_\sQ)$.

To show  that  $(\sF, \nabla_\sQ)$ is totally disconnected, let $\hat F \in \sF$ and $\check F \in \sF \setminus \{\hat F\}$.
Then as in the proof of Corollary \ref{tildehat} there are compact sets $\hat H$ and $\check H$ in  $\mF$ with
$$  \mF = \hat H \cup \check H,\quad
\hat H\cap \check H = \emptyset,\quad \hat F \subset \hat H \text{ and }\check F \subset \check H
$$
and every  element of $\sF$ is a subset  either of $\hat H$ or of $\check H$. Therefore 
$$\pi_\sQ(\hat H) \colon= \{F \in \sF: F \subset \hat H\} \text{ and  }\pi_\sQ(\check H) = \{F \in \sF: F \subset \check H\}$$ 
are closed in $(\sF, \nabla_\sQ)$, $ \pi_\sQ(\hat H) \cup \pi_\sQ(\check H) = \sF \text{ and }  \pi_\sQ(\hat H) \cap \pi_\sQ(\check H) = \emptyset$. Since
$\hat F \in  \pi_\sQ(\hat H)$ and $\check F \in \pi_\sQ(\check H)$,
 by Remark \ref{discon} $\sF$ is totally disconnected.
\end{proof}    

\begin{theorem}\label{L0w} For every $x \in \mO$ there is a neighbourhood $U_x \subset \mO$ such that
   $$ d(x_1,x_2) = \nabla_\sQ  \big(\{x_1\},\{x_2\}\big) \forall x_1,\,x_2 \in U_x.$$
\end{theorem}
\begin{proof} This is immediate from Corollary \ref{Cxy} and \eqref{nab2}.
\end{proof} 
\begin{theorem}\label{Peanoresult}
$({\sQ},\nabla_\sQ)$ is a Peano continuum.
\end{theorem}
\begin{proof}   
Since a continuum is a Peano continuum   if and only if it has no congestion points, by Theorem\ref{comcon}  it suffices to prove that $({\sQ},\nabla_\sQ)$ is weakly locally connected at every  $Q \in \sQ$. 

For  weak local connectedness of $\sQ$ at $Q = \{x\}\in \sO$,  as in Corollary \ref{Cxy} let $0<\d_x < \dist(x,\mF)/3$. Then since $x\in \mO$ is not a congestion point of $M$, for any $\e \in (0, \d_x)$ there is a  connected neighbourhood,  $W_\e$ say, of $x$ in $(M,d)$ with  $x\in W_\e \subset B_\e(x)$. Then by Corollary \ref{Cxy} $\Delta_\sP(x,y) = d(x,y)$ for all $y \in W_\e$ and it follows that $ \pi_\sQ(W_\e)$  is a  connected neighbourhood of $\{x\}$  in the ball $ B_\e\big(\{x\}\big)$ in $(\sQ,\nabla_\sQ)$. Thus $(\sQ, \nabla_\sQ)$ is weakly locally connected at all points of $\sO$. Hence all the congestion points, if any, of $(\sQ,\nabla_\sQ)$ are in $\sF$.
However, since by Theorem \ref{Fdiscon} $\sF$ is totally disconnected, it follows from  Theorem \ref{whyburn} that no point of $\sF$ is a congestion point of  $(\sQ,\nabla_\sQ)$. Hence $(\sQ,\nabla_\sQ)$ is weakly locally connected and the proof is complete.
\end{proof}
\section{Closing Remarks}\label{closing} 
For any metric space $M$ it follows  from Section \ref{SPM} that  if $\mF=\ol{\sN(M)}$  is compact  and $\sP$, $\sF$ and $\sO$ are defined by \eqref{FO}, $\sP = \sQ$ and  the quotient space $(\sQ,\nabla_Q)$ is a Peano metric space (the paragraph after Definition \ref{loccon}).

Suppose a  closed subset $\mF^*$ of a continuum $M$  contains all the congestion points of $M$. Then if $\mO^* = M \setminus \mF^*$, a partition   $\sP^*=\sO^*\cup\sF^*$ of $M$ is defined  by replacing $\mF$ with $\mF^*$ in \eqref{F&O}. Then  
by the above arguments the analogous quotient metric space $(\sQ^*, \nabla_{\sQ^*})$  is a Peano continuum with $\sP^* = \sQ^*$. Moreover,   $\sF^*$ is compact and  totally disconnected, and the quotient metric $\nabla_{\sQ^*}$ coincides locally with the original metric $d$ in $\mO^*$, in the  sense  of Theorem \ref{L0w}.

An example of a closed set  which contains the congestion points of $M$ is the closure of the set of points where $M$ is not locally connected, Definition \ref{loccon}.

On the other hand if $M = \mO^\dagger \cup \mF^\dagger$ where $\mF^\dagger$ is closed, $\mO^\dagger \cap \mF^\dagger= \emptyset$, but $\mO^\dagger \cap\sN(M) \neq \emptyset$, it follows from Corollary \ref{Cxy} with $\mO^\dagger$ and $\mF^\dagger$ instead of $\mO$ and $\mF$ that 
for every $x \in \mO^\dagger$ there is a neighbourhood $U^\dagger_x \subset \mO^\dagger$ such that
 $d(x_1,x_2) = \nabla_{\sQ^\dagger}  \big(\{x_1\},\{x_2\}\big) \forall x_1,\,x_2 \in U^\dagger_x$.
 
Therefore $\{x\} \in \sO^\dagger$ is a congestion point of $(\sQ^\dagger, \nabla_{\sQ^\dagger})$ if $x \in \sN(M) \cap \mO^\dagger$,  and hence 
$(\sQ^\dagger, \nabla_{\sQ^\dagger})$ is compact and connected, but  not a Peano continuum.  

Therefore the choice  $\sF = \ol{\sN(M)}$ in Section \ref{PQMST} is optional.
 \qed

        \end{document}